\documentclass[reqno]{article}
\usepackage{latexsym}
\usepackage{amsmath}
\usepackage{amssymb}
\usepackage{amsthm}
\usepackage{amscd}

\newtheorem{theorem}{Theorem}
\newtheorem{lemma}{Lemma}
\newtheorem{remark}{Remark}
\newtheorem{definition}{Definition}

\newtheorem{corollary}{Corollary}

\begin{document}

\title{An intrinsic square function on weighted Herz spaces with variable exponent}

\date{}

\maketitle

\begin{center}
{\large Mitsuo Izuki{\small *}} 
and {\large  Takahiro Noi{\small **}}
\end{center}

\begin{enumerate}
    \item[*]{\it Faculty of Education, \\
Okayama University, \\ 
3-1-1 Tsushima-naka, Okayama 700-8530, Japan \\
    {
     e-mail: izuki@okayama-u.ac.jp}}
    \item[**]{\it Department of Mathematics and Information Science, \\
Tokyo Metropolitan University, \\ 
Hachioji, 192-0397, Japan \\ 
    {
     e-mail: taka.noi.hiro@gmail.com}} 
\end{enumerate}

\begin{abstract}
We define new generalized Herz spaces having weight and variable exponent, 
that is, weighted Herz spaces with variable exponent. 
We prove the boundedness of an intrinsic square function on those spaces 
under proper assumptions on each exponent and weight. 
\end{abstract}

\vspace{1.5cm}
\noindent{\bf Keywords}\ \ 
Herz spaces, Muckenhoupt weight, variable exponent, intrinsic square function.

\noindent{\bf $2010$ Mathematics Subject Classification}\ \ 42B35

\section{Introduction}
\noindent

The boundedness of the Hardy--Littlewood maximal operator $M$ 
on function spaces is very important in real analysis
because it realizes 
boundedness of many other operators, 
for example, singular integrals, fractional integrals, 
and commutators involving {\rm BMO} functions. 
Muckenhoupt \cite{Muckenhoupt1972} 
has established the theory on weights called the Muckenhoupt $A_p$ theory 
in the study of weighted function spaces and 
greatly developed real analysis. 

On the other hand, the theory on function spaces with variable exponent 
has been rapidly developed after the work \cite{KR} 
where Kov\'a$\mathrm{ \check{c}}$ik and R\'akosn\'ik 
have clarified fundamental properties of Lebesgue spaces with variable exponent. 
On spaces with variable exponent, the boundedness of the operator $M$ is 
also a notable problem. 
As sufficient conditions for the boundedness 
the log-H\"older continuous conditions have been established and 
well-known now (\cite{CFN1,CFN2,Diening2004,Diening2005,INS-SMJ}). 

Recently a generalization of the Muceknhoupt weights in terms of variable exponent 
has been studied. 
Diening and H\"ast\"o \cite{DH-preprint} have initially defined 
the new class of weights $A_{p(\cdot)}$ and 
proved the equivalence between the $A_{p(\cdot)}$ condition and
the boundedness of $M$ on weighted Lebesgue spaces with variable exponent. 
The equivalence has been independently proved by 
Cruz-Uribe, Fiorenza and Neugebauer \cite{CFN2012}. 
The first author and his collaborators have studied 
the relation between $A_{p(\cdot)}$ and the wavelet theory 
(\cite{IzukiJAA2013,INS-weighted}). 

An intrinsic square function is one of the remarkable operators in modern real analysis. 
Many researchers have studied 
characterizations of general function spaces via 
intrinsic functions (\cite{ref1,ref2,ref3,Wang2014,ref4,ref5,Wilson,ref6}). 
In particular we focus on the work \cite{Wang2014} by Wang 
where the boundedness of 
 some intrinsic functions including $S_{\beta}$ 
on weighted Herz spaces has been proved under proper assumptions on 
every exponent and weight. 
Our aim in this paper is to extend the boundedness of 
the intrinsic square function $S_{\beta}$ to 
the variable exponent case. 
We will define weighted Herz spaces with variable exponent
having variable integral exponent $p(\cdot)$ and weight $w$, 
and prove the boundedness of $S_{\beta}$ on those spaces 
based on fundamental facts on general Banach function spaces. 

Throughout this paper we will use the following notation.

\begin{enumerate}
\item 
The symbol $C$ always denotes a positive constant 
independent of main parameters. We remark that 
the value of $C$ may be different 
from one occurrence to another.
\item
Given a measurable set $S\subset \mathbb{R}^n$, 
we denote the Lebesgue measure of $S$ by $|S|$.
In addition, $\chi_S$ means the characteristic function of $S$. 
\item 
A ball is always an open ball in $\mathbb{R}^n$, that is, 
a ball $B$ is a set given by 
\[
B:=\{  y\in \mathbb{R}^n \, : \, |x-y|<r \} ,
\]
using a point $x\in \mathbb{R}^n$ and a positive number $r$. 
\end{enumerate}


\section{Preliminaries}

\subsection{Lebesgue spaces with variable exponent}
\noindent

We first define Lebesgue space with variable exponent.

\begin{definition}
Let $p(\cdot) \, : \, \mathbb{R}^n \to [1,\, \infty)$ be a measurable function. 
The Lebesgue space $L^{p(\cdot)}(\mathbb{R}^n)$ with variable exponent $p(\cdot)$
is  
the set of all complex-valued measurable functions $f$ defined on $\mathbb{R}^n$ 
satisfying
\[
\rho_p(f):= \int_{\mathbb{R}^n} |f(x)|^{p(x)}dx <\infty . 
\]
\end{definition}

It is known (cf. \cite{INS-SMJ,KR}) that the Lebesgue space $L^{p(\cdot)}(\mathbb{R}^n)$ 
becomes a Banach space equipped with a norm given by
\[
\| f \|_{L^{p(\cdot)}(\mathbb{R}^n)}:=
\inf \left\{  \lambda >0 \, : \, 
\rho_p\left( \frac{f}{\lambda} \right) \le 1 \right\}  .
\]
The measurable function $p(\cdot)$ is called a variable exponent 
in variable exponent analysis. 
In order to state variable exponent spaces deeply 
we define some notations on variable exponents. 

\begin{definition}
\begin{enumerate}
\item 
Given a measurable function $r(\cdot) \, : \, \mathbb{R}^n \to (0,\, \infty)$, we write 
\[
r_+:= \| r(\cdot) \|_{L^{\infty}(\mathbb{R}^n)}, \quad 
r_-:= \left\{  \left( \frac{1}{r(\cdot)} \right)_+ \right\}^{-1} .
\]
\item
The set $\mathcal{P}(\mathbb{R}^n)$ consists of all variable exponents
$p(\cdot)\, : \, \mathbb{R}^n \to [1,\, \infty)$ satisfying 
$1<p_-\le p_+<\infty$. 
\item
A measurable function $r(\cdot) \, : \, \mathbb{R}^n \to (0,\, \infty)$ 
is said to be globally log-H\"older continuous if it satisfies the following two inequalities
\begin{eqnarray*}
|r(x)-r(y)|
&\le & \frac{C}{-\log (|x-y|)} \quad (|x-y| \le 1/2), \\
|r(x)-r_{\infty}|
&\le&
\frac{C}{\log(e+|x|)} \quad (x\in \mathbb{R}^n)
\end{eqnarray*}
for some real constant $r_{\infty}$.
The set $LH(\mathbb{R}^n)$ consists of all globally log-H\"older continuous 
functions. 
\end{enumerate}
\end{definition}

The globally log-H\"older continuous conditions are famous 
because they ensure the boundedness of the Hardy--Littlewood maximal operator $M$, 
defined by
\[
Mf(x):=\sup_{B\, : \, \mathrm{ball}, \, x\in B}
\int_B |f(y)|\, dy ,
\]
on Lebesgue spaces with variable exponent.
Hence we often consider those conditions as 
standard assumptions in the study of 
function spaces with variable exponents 
(cf. \cite{CFN1,CFN2,Diening2004,Diening2005,INS-SMJ}).


\subsection{Weighted Banach function spaces}
\noindent

We define Banach function space 
and state fundamental properties of it based on the book \cite{BS-Book} by  
Bennett and Sharpley.
For further informations on the theory of Banach function space 
including the proof of Lemma \ref{Lemma-Banach-associate}
 below we refer to the book. 
We additionally show some properties 
of Banach function spaces in terms of boundedness of
the Hardy--Littlewood maximal operator. 
We will also consider the weighted case 
based on the paper \cite{KarlovichSpitkovsky} by 
Karlovich and Spitkovsky.

\begin{definition}
Let $\mathcal{M}$ be the set of all complex-valued measurable functions 
defined on $\mathbb{R}^n$, and 
$X$ a linear subspace of $\mathcal{M}$. 
\begin{enumerate}
\item
The space $X$ is said to be a Banach function space if there exists a functional 
$\| \cdot \|_X \, : \, \mathcal{M} \to [0,\infty]$ 
satisfying the following properties: 
Let $f, \, g, \, f_j \in \mathcal{M}$ $(j=1, \, 2, \, \cdots)$, then 
\begin{enumerate} 
\item
$f\in X$ holds if and only if $\| f \|_X<\infty$.
\item Norm property:
  \begin{enumerate}
  \item Positivity: $\| f \|_X \ge 0$.
  \item Strict positivity: $\| f \|_X=0$ holds if and only if $f(x)=0$ for almost every 
                                  $x\in \mathbb{R}^n$.
  \item Homogeneity: $\| \lambda f \|_X=|\lambda| \cdot \| f \|_X$ holds for all 
                              complex numbers $\lambda$.
  \item Triangle inequality: $\| f+g \|_X \le \| f \|_X +\| g \|_X$.
  \end{enumerate}

\item Symmetry: $\| f \|_X= \| |f| \|_X$.
\item Lattice property: If $0\le g(x)\le f(x)$ for almost every $x\in \mathbb{R}^n$, 
         then $\| g \|_X \le \| f \|_X$.
\item Fatou property: 
If $0\le f_j(x)\le f_{j+1}(x)$ for all $j$ 
and $f_j(x)\to f(x)$ as $j\to \infty$
for 
almost every $x\in \mathbb{R}^n$, then 
$\displaystyle{\lim_{j\to \infty} \| f_j \|_X  =\| f \|_X}$. 

\item
For every measurable set $F\subset \mathbb{R}^n$ such that $|F|<\infty$, 
$\| \chi_F \|_X$ is finite. Additionally there exists a constant $C_F>0$ 
depending only on $F$ such that for all $h\in X$, 
\[
\int_F |h(x)|\, dx \le C_F \| h \|_X.
\]
\end{enumerate}

\item
Suppose that $X$ is a Banach function space 
equipped with a norm $\| \cdot \|_X$.
The associated space $X'$ is defined by
\[
X':=\{  f\in \mathcal{M} \, : \, \| f \|_{X'}<\infty  \} ,
\]
where
\[
\| f \|_{X'}:=\sup_g \left\{
\left| \int_{\mathbb{R}^n} f(x)g(x)\, dx \right|
\, : \, \| g \|_X\le 1
\right\} .
\]
\end{enumerate}
\end{definition}

\begin{lemma}
\label{Lemma-Banach-associate}
Let $X$ be a Banach function space. 
Then the following hold:
\begin{enumerate}
\item
The associated space $X'$ is also a Banach function space.

\item (The Lorentz--Luxemberg theorem. ) 
$(X')'=X$ holds, in particular, 
the norms $\| \cdot \|_{(X')'}$ and $\| \cdot \|_X$ are 
equivalent. 
\item (The generalized H\"older inequality. ) 
If $f\in X$ and $g\in X'$, then we have
\[ \int_{\mathbb{R}^n} |f(x)g(x)|\, dx
\le \| f \|_X \| g \|_{X'}.
\]
\end{enumerate}
\end{lemma}

Kov$\acute{\rm{a}}\check{\rm{c}}$ik and R$\acute{\rm{a}}$kosn\'ik 
\cite{KR} have proved that 
the generalized Lebesgue space
$L^{p(\cdot)}(\mathbb{R}^n)$ with variable exponent 
$p(\cdot)$
is a Banach function space and the associate space is 
$L^{p'(\cdot)}(\mathbb{R}^n)$
with norm equivalence, where $p'(\cdot)$ is the conjugate exponent given 
by $\frac{1}{p(\cdot)} +\frac{1}{p'(\cdot)}=1$.

If we assume some conditions for boundedness of the 
Hardy--Littlewood maximal operator $M$ on X, 
then the norm $\| \cdot \|_X$ has properties similar to 
the Muckenhoupt weights.

\begin{lemma}
\label{Lemma-Izuki-Banach}
Let $X$ be a Banach function space. 
Suppose that the Hardy--Littlewood maximal operator 
$M$ is weakly bounded on $X$, that is, 
\begin{equation}
\| \chi_{\{ Mf>\lambda \} } \|_{X}
\le C\lambda^{-1} \| f \|_X
\label{weakly bounded}
\end{equation}
is true for all $f\in X$ and $\lambda >0$.
Then we have 
\begin{equation}
\sup_{B:\mathrm{ball}} 
\frac{1}{|B|} \| \chi_B \|_X \| \chi_B \|_{X'} < \infty .
\label{Banach-Muckenhoupt}
\end{equation}
\end{lemma}

The proof of Lemma \ref{Lemma-Izuki-Banach} is found in 
the first author's paper 
\cite[Lemmas 2.4 and 2.5]{IzukiJAA2013}
and \cite[Lemmas G' and H]{IST2014}. 

\begin{remark}
\label{Remark-Banach}
If $M$ is bounded on $X$, that is, 
\[
\| Mf \|_X\le C\, \| f \|_X 
\]
holds for all $f\in X$, then 
we can easily check that (\ref{weakly bounded}) holds. 
On the other hand, if $M$ is bounded on the associate space $X'$, 
then Lemma \ref{Lemma-Banach-associate} shows that 
(\ref{Banach-Muckenhoupt}) is true. 
\end{remark}

Below we define weighted Banach function space
and give some properties of it. 
Let $X$ be a Banach function space. 
The set $X_{\mathrm{loc}}(\mathbb{R}^n)$ consists of all measurable function $f$ 
such that $f\chi_E \in X$ for any compact set $E$ with $|E|<\infty$.  
Given a function $W$ such that 
$0<W(x)<\infty$ for almost every $x\in \mathbb{R}^n$, 
$W\in X_{\mathrm{loc}}(\mathbb{R}^n)$ 
and $W^{-1} \in (X')_{\mathrm{loc}}(\mathbb{R}^n)$, 
we define the weighted Banach function space 
\[
X(\mathbb{R}^n ,W):=
\left\{ f\in \mathcal{M}
\, : \, fW \in X
\right\} .
\]
Then the following hold.

\begin{lemma} \
\label{Lemma-weighted-Banach}
\begin{enumerate}
\item
The weighted Banach function space $X(\mathbb{R}^n , W)$ is a Banach function space 
equipped the norm
\[
\| f \|_{X(\mathbb{R}^n ,W)}:=\| f W \|_{X} .
\]

\item
The associate space of $X(\mathbb{R}^n ,W)$ is also a Banach function space 
and equals to $X'(\mathbb{R}^n ,W^{-1})$. 
\end{enumerate}
\end{lemma}

The properties above naturally arise from those of usual Banach function spaces 
and the proof is found in \cite{KarlovichSpitkovsky}.


\subsection{Muckenhoupt weights with variable exponent}
\noindent

A locally integrable and positive function defined on $\mathbb{R}^n$ 
is called a weight. 
We define 
fundamental classes of weights 
known as the Muckenhoupt classes 
in terms of variable exponent.

\begin{definition}
\label{def-variable-Muckenhoupt}
Suppose $p(\cdot) \in \mathcal{P}(\mathbb{R}^n)$. 
A weight $w$ is said to be an $A_{p(\cdot)}$ weight if 
\[
\sup_{B:\mathrm{ball}} \frac{1}{|B|} \| w^{1/p(\cdot)}\chi_B \|_{L^{p(\cdot)}(\mathbb{R}^n)}
\| w^{-1/p(\cdot)}\chi_B \|_{L^{p'(\cdot)}(\mathbb{R}^n)}
<\infty .
\]
The set $A_{p(\cdot)}$ consists of all $A_{p(\cdot)}$ weights.
\end{definition}

\begin{remark}
Our symbol $A_{p(\cdot)}$ 
differs from that in the papers \cite{CFN2012,DH-preprint}.
If $p(\cdot)\in \mathcal{P}(\mathbb{R}^n)$ equals to a constant 
$p\in (1, \, \infty)$, then 
the definition above is equivalent to the well known Muckenhoupt $A_p$ weights 
(\cite{Muckenhoupt1972}). 
\end{remark}

We shall give the definitions of the Muckenhoupt classes 
$A_p$ with $p=1, \, \infty$.

\begin{definition}
\begin{enumerate}
\item
A weight $w$ is said to be a Muckenhoupt $A_1$ weight if 
$Mw(x)\le C\, w(x)$ holds for almost every $x\in \mathbb{R}^n$. 
The set $A_1$ consists of all Muckenhoupt $A_1$ weights. 
For every $w\in A_1$, the finite value 
\[
[w]_{A_1}:=\sup_{B:\mathrm{ball}}
\left\{
\frac{1}{|B|} \int_B w(x)\, dx
\cdot \| w^{-1} \|_{L^{\infty}(B)}
\right\}
\]
is said to be a Muckenhoupt $A_1$ constant.
\item
A weight belonging to the set 
\[
A_{\infty}:=\bigcup_{1<p<\infty}A_p
\]
is said to be a Muckenhoupt $A_{\infty}$ weight. 
\end{enumerate}
\end{definition}

\begin{remark}
\begin{enumerate}
\item
We note that if $w\in A_1$, then 
\[ \frac{1}{|B|} \int_B w(x)\, dx
\le [w]_{A_1} \inf_{x\in B} w(x)
\]
holds for all balls $B$. 
\item
It is known that the monotone property 
$A_p \subset A_q \subset A_{\infty}$
holds for every constants $1\le p<q<\infty$.
\end{enumerate}
\end{remark}

We will use a classical result on the Muckenhoupt weights. 
Below we write 
$$w(S):=\int_S w(x) \, dx$$
for a measurable set $S$ and a weight $w$. 

\begin{lemma}[Chapter 7 in \cite{Duoa-Book}]
\label{Lemma-Muckenhoupt}
If $w\in A_1$, then
there exist positive constants 
$\delta<1$ and $C$ depending only on $n$ and $[w]_{A_1}$ 
such that for all balls $B$ and all measurable sets $E\subset B$, 
\[
\frac{w(E)}{w(B)} \le C
\left( \frac{|E|}{|B|} \right)^{\delta} .
\]
\end{lemma}

Diening and H\"ast\"o \cite{DH-preprint}
have pointed out 
that Definition \ref{def-variable-Muckenhoupt}
does not directly imply the monotone property 
of the Muckenhoupt class $A_{p(\cdot)}$. 
In order to obtain the property they 
have generalized the Muckenhoupt class as follows:

\begin{definition}[Diening and H\"ast\"o \cite{DH-preprint}]
Suppose $p(\cdot) \in \mathcal{P}(\mathbb{R}^n)$. 
A weight $w$ is said to be an $\tilde{A}_{p(\cdot)}$ weight if 
\[
\sup_{B:\mathrm{ball}} 
|B|^{-p_B} \| w\chi_B\|_{L^1(\mathbb{R}^n)} 
\| w^{-1}\chi_B \|_{L^{p'(\cdot)/p(\cdot)}(\mathbb{R}^n)}
<\infty ,
\]
where $p_B$ is the harmonic average of $p(\cdot)$ over $B$, namely, 
\[
p_B:=\left( \frac{1}{|B|} \int_B \frac{1}{p(x)} \, dx  
\right)^{-1} .
\]
The set $\tilde{A}_{p(\cdot)}$ consists of all $\tilde{A}_{p(\cdot)}$ weights.
\end{definition}

Based on the definition $\tilde{A}_{p(\cdot)}$
Diening and H\"ast\"o 
\cite[Lemma 3.1]{DH-preprint} 
have proved the next monotone property.

\begin{theorem}
\label{Theorem-DH-monotone}
Suppose $p(\cdot) , \, q(\cdot)\in \mathcal{P}(\mathbb{R}^n)\cap LH(\mathbb{R}^n)$ 
and $p(\cdot) \le q(\cdot)$. 
Then we have
\[
A_1 \subset A_{p_{-}} \subset \tilde{A}_{p(\cdot)}
\subset \tilde{A}_{q(\cdot)}
\subset A_{q_+} \subset A_{\infty} . 
\]
\end{theorem}

Before we state the relation between 
the generalized Muckenhoupt condition and 
boundedness of the Hardy--Littlewood maximal operator, 
we define explicitly weighted Lebesgue spaces with variable exponent. 

\begin{definition}
Let $p(\cdot) \in \mathcal{P}(\mathbb{R}^n)$ and 
$w$ be a weight. 
The weighted Lebesgue space with variable exponent $L^{p(\cdot)}(w)$ 
is defined by
\[
L^{p(\cdot)}(w):=L^{p(\cdot)}(\mathbb{R}^n, \, w^{1/p(\cdot)}) .
\]
Namely the space $L^{p(\cdot)}(w)$ 
is a Banach function space equipped with the norm
\[
\| f \|_{L^{p(\cdot)}(w)}:=\| fw^{1/p(\cdot)} \|_{L^{p(\cdot)}(\mathbb{R}^n)} .
\]
\end{definition}

\begin{theorem}
\label{Theorem-Muckenhoupt-variable-exponent}
Suppose $p(\cdot) \in \mathcal{P}(\mathbb{R}^n)\cap LH(\mathbb{R}^n)$. 
Then the following three conditions are equivalent:
\begin{enumerate}
\item[(A)]
$w\in A_{p(\cdot)}$.
\item[(B)]
$w\in \tilde{A}_{p(\cdot)}$.
\item[(C)]
The Hardy--Littlewood maximal operator is bounded 
on the weighted variable Lebesgue space $L^{p(\cdot)}(w)$.
\end{enumerate}
\end{theorem}

Cruz-Uribe,  Fiorenza and Neugebauer \cite{CFN2012}
have proved $(A)\Leftrightarrow (C)$.
On the other hand, 
Diening and H\"ast\"o \cite{DH-preprint} 
have proved $(B)\Leftrightarrow (C)$.
By Theorem \ref{Theorem-Muckenhoupt-variable-exponent} we can
identify $A_{p(\cdot)}$ and $\tilde{A}_{p(\cdot)}$, 
provided that $p(\cdot) \in \mathcal{P}(\mathbb{R}^n)\cap LH(\mathbb{R}^n)$. 
Moreover Theorem \ref{Theorem-DH-monotone} gives us
the following monotone property.

\begin{corollary}
Suppose $p(\cdot) , \, q(\cdot)\in \mathcal{P}(\mathbb{R}^n)\cap LH(\mathbb{R}^n)$ 
and $p(\cdot) \le q(\cdot)$. 
Then we have
\[
A_1 \subset A_{p_{-}} \subset A_{p(\cdot)}
\subset A_{q(\cdot)}
\subset A_{q_+} \subset A_{\infty} . 
\]
\end{corollary}

\begin{remark}
Based on 
Diening and H\"ast\"o 
\cite[Proposition 3.1]{DH-preprint}, 
we can construct weights belonging to 
$A_{p(\cdot)}$ as follows: 
Let $p(\cdot)\in \mathcal{P}(\mathbb{R}^n)\cap LH(\mathbb{R}^n)$ and 
$w_1, \, w_2\in A_1$.
Then we have that $w_1 w_2^{1-p(\cdot)} \in A_{p(\cdot)}$.
\end{remark}



\section{Main result}

\subsection{Definition of the intrinsic function and Herz spaces}

\noindent 

We first define  the 
intrinsic square function $S_{\beta}f(x)$.

\begin{definition}
Given a point $x\in \mathbb{R}^n$, we define a set
\[
\Gamma(x):= \{ (y,t)\in \mathbb{R}^{n+1}_+ \, : \, |x-y|<t \} ,
\]
where $\mathbb{R}^{n+1}_+=\mathbb{R}^n \times (0,\infty)$.
Let $0<\beta \le 1$ be a constant.
The set $\mathcal{C}_{\beta}$ consists of all functions $\varphi$
defined on $\mathbb{R}^n$ such that
\begin{enumerate}
\item $\mathrm{supp} \varphi \subset \{ |x|\le 1 \}$, 
\item $\int_{\mathbb{R}^n} \varphi(x)\, dx=0$,
\item $|\varphi(x)-\varphi(x')|\le |x-x'|^{\beta}$ for $x, \, x'\in \mathbb{R}^n$.
\end{enumerate}
For every $(y,t)\in \mathbb{R}^{n+1}_+$ we write 
$\varphi_t(y)=t^{-n}\varphi\left( \frac{y}{t}\right)$.
Then we define a maximal function for $f\in L^1_{\mathrm{loc}}(\mathbb{R}^n)$, 
\[
A_{\beta}f(y,t):=\sup_{\varphi \in \mathcal{C}_{\beta}} |f*\varphi_t(y)|
\quad \left( (y,t)\in \mathbb{R}^{n+1}_+ \right) .
\]
Using above, we define the 
intrinsic square function with order $\beta$ by 
\[
S_{\beta}f(x):=\left(
\int \! \int_{\Gamma(x)} 
A_{\beta}f(y,t)^2 \, \frac{dy\, dt}{t^{n+1}}
\right)^{1/2} .
\]
\end{definition}

In order to define weighted Herz spaces with variable exponent, 
we use a local weighted Lebesgue spaces with variable exponent.

\begin{definition}
Let $\Omega \subset \mathbb{R}^n$ be a measurable set,
$p(\cdot) \, : \, \Omega \to [1,\infty)$  a measurable function
 and 
$w$  a positive and locally integrable function defined on $\Omega$. 
The set $L^{p(\cdot)}_{\mathrm{loc}}(\Omega ,w^{1/p(\cdot)})$ consists of all functions $f$
satisfying the following condition: 
for all measurable subsets $E\subset \Omega$
there exists a constant $\lambda >0$ such that
\[
\int_E \left|  \frac{f(x)}{\lambda} \right|^{p(x)} w(x)\, dx<\infty .
\]
\end{definition}

We additionally use the following notation.

\begin{enumerate}
\item
For every integer $k$, we write 
$B_k:=\{ |x|\le 2^k \}$, 
$D_k:=B_k \, \backslash \, B_{k-1}$ and $\chi_k:=\chi_{D_k}$.
\item
For every non-negative integer $m$, we write
$C_m:=D_m$ if $m \ge 1$ and $C_0:=B_0$. 
\end{enumerate}

Now we are ready to define the Herz spaces.

\begin{definition}
Let $\alpha \in \mathbb{R}$, $0<q<\infty$, $p(\cdot)\in \mathcal{P}(\mathbb{R}^n)$ and
$w$ be a weight. 
\begin{enumerate}
\item
The homogeneous weighted Herz space $\dot{K}^{\alpha,q}_{p(\cdot)}(w)$ 
with variable exponent is defined by
\[
\dot{K}^{\alpha,q}_{p(\cdot)}(w):=\{   
f\in L^{p(\cdot)}_{\mathrm{loc}}(\mathbb{R}^n \, \backslash \, \{ 0 \} ,w^{1/p(\cdot)})
\, : \, 
\| f \|_{\dot{K}^{\alpha,q}_{p(\cdot)}(w)}<\infty
\}  ,
\]
where 
\[
\| f \|_{\dot{K}^{\alpha,q}_{p(\cdot)}(w)}
:=\left( \sum_{k=-\infty}^{\infty} 2^{\alpha k q}
\| f\chi_k \|_{L^{p(\cdot)}(w)}^q
\right)^{1/q} .
\]

\item
The non-homogeneous weighted Herz space $K^{\alpha,q}_{p(\cdot)}(w)$ 
with variable exponent is defined by
\[
K^{\alpha,q}_{p(\cdot)}(w):=\{   
f\in L^{p(\cdot)}_{\mathrm{loc}}(\mathbb{R}^n  ,w^{1/p(\cdot)})
\, : \, 
\| f \|_{K^{\alpha,q}_{p(\cdot)}(w)}<\infty
\}  ,
\]
where 
\[
\| f \|_{K^{\alpha,q}_{p(\cdot)}(w)}
:=\left( \sum_{m=0}^{\infty} 2^{\alpha m q}
\| f\chi_{C_m} \|_{L^{p(\cdot)}(w)}^q
\right)^{1/q} .
\]

\end{enumerate}
\end{definition}

\subsection{Key lemmas}

\noindent

Lemmas \ref{Lemma-1} and \ref{Lemma-2} below 
have been proved by the first author \cite[Proposition 2.4]{IzukiAnalMath2010}
in the case $X=L^{p(\cdot)}(\mathbb{R}^n)$. 
His proof of Lemma \ref{Lemma-2} is due to Diening's work \cite{Diening2005}.
Recently a self-contained proof based on the Rubio de Francia 
algorithm \cite{Rubio1,Rubio2,Rubio3} 
has given by 
Cruz-Uribe, Hern\'andez and Martell \cite[Proof of Lemma 3.3]{CHM}.
Based on \cite{IzukiAnalMath2010,CHM} 
we will give the complete proofs of those lemmas.

\begin{lemma}
\label{Lemma-1}
Let $X$ be a Banach function space. 
Suppose that the Hardy--Littlewood maximal operator $M$
is weakly bounded on $X$. 
Then we have that for all balls $B\subset \mathbb{R}^n$ 
and all measurable sets $E\subset B$, 
\begin{equation}
\frac{|E|}{|B|}
\le C\,
\frac{\| \chi _E \|_X}{\| \chi_B \|_X}  .
\label{equ-lemma1}
\end{equation}
\end{lemma}
\begin{proof}
Take a ball $B$, a measurable set $E\subset B$ 
and a number $0<\lambda<|E|/|B|$ arbitrarily. 
Then we see that $M(\chi_E)(x)>\lambda$ for almost every $x\in B$. 
Hence we have
\[
\| \chi_B \|_X \le  \| \chi_{\{  M(\chi_E)>\lambda\} }\|_X
\le C\, \lambda^{-1} \, \| \chi_E \|_X .
\]
Therefore we get inequality (\ref{equ-lemma1}) because 
$0<\lambda<|E|/|B|$ is arbitrary.
\end{proof}

\begin{lemma}
\label{Lemma-2}
Let $X$ be a Banach function space. Suppose that 
$M$ is bounded on the associate space $X'$. 
Then there exists  a constant $0<\delta <1$ such that for all balls $B\subset \mathbb{R}^n$ 
and all measurable sets $E\subset B$, 
\begin{equation}
\frac{\| \chi _E \|_X}{\| \chi_B \|_X}  
\le C\,
\left(
\frac{|E|}{|B|}
\right)^{\delta}.
\label{equ-lemma2}
\end{equation}
\end{lemma}

\begin{proof}
Let $A:= \| M \|_{X' \to X'}$ and 
define a function 
\begin{equation}
Rg(x):= \sum_{k=0}^{\infty} \frac{M^kg(x)}{(2A)^k} \quad (g\in X')
,
\label{define Rubio}
\end{equation}
where
\[
M^kg:= \begin{cases}
|g| & (k=0), \\
Mg & (k=1), \\
M(M^{k-1}g) & (k\ge 2).
\end{cases}
\]
For every $g\in X$, the function $Rg$ satisfies the following properties:
\begin{enumerate}
\item
$|g(x)| \le Rg(x)$ for almost every $x\in \mathbb{R}^n$.
\item
$\| Rg \|_{X'}\le 2\| g\|_{X'}$, namely the operator $R$ is bounded on $X'$.
\item
$M(Rg)(x) \le 2A Rg(x)$, that is, $Rg$ is a Muckenhoupt $A_1$ weight 
such that $[Rg]_{A_1}\le 2A$. 
\end{enumerate}
Thus by applying Lemma \ref{Lemma-Muckenhoupt} to $Rg$,  
we can take positive constants $C$ and $\delta <1$ 
so that for all balls $B$ and all measurable sets $E\subset B$, 
\[
\frac{Rg(E)}{Rg(B)} \le C\left( \frac{|E|}{|B|} \right)^{\delta} .
\]
Now we fix $g\in X'$ with $\| g \|_{X'} \le 1$ arbitrarily. 
By virtue of generalized H\"older's inequality we have
\begin{eqnarray*}
\int_{  \mathbb{R}^n }
|\chi_E(x) g(x)| \, dx
&\le &
Rg(E)  \\
&\le &
C \left( \frac{|E|}{|B|} \right)^{\delta}
\cdot Rg(B)
\\
&\le &
C \left( \frac{|E|}{|B|} \right)^{\delta}
\cdot \| \chi_B \|_{X} \| Rg \|_{X'}
\\
&\le &
C \left( \frac{|E|}{|B|} \right)^{\delta}
 \| \chi_B \|_{X} .
\end{eqnarray*}
Therefore by the duality we get 
\begin{eqnarray*}
\| \chi_E \|_{X}
& \le & 
C\, \sup_g \left\{ \left|  
\int_{  \mathbb{R}^n }
\chi_E(x) g(x) \, dx \right|  
\, : \, g\in X', \ \| g \|_{X'} \le 1
\right\} 
\\
&\le &
C \left( \frac{|E|}{|B|} \right)^{\delta}
 \| \chi_B \|_{X} .
\end{eqnarray*}
This completes the proof of the lemma.
\end{proof}

Wilson \cite{Wilson} has proved 
the following boundedness of the square function on weighted Lebesgue spaces.

\begin{theorem}
\label{Theorem-Wilson}
Let $0<\beta \le 1$, $1<p<\infty$ and $w\in A_p$. 
Then the square function $S_\beta$ is bounded on the weighted 
Lebesgue space $L^p(w)$.
\end{theorem}

The next extrapolation theorem 
on weighted Lebesgue spaces
has recently proved by
Cruz-Uribe and Wang \cite[Theorem 2.6]{CruzWang-arxiv}.

\begin{theorem}
\label{Theorem-weighted-extrapolation}
Suppose that 
there exists a constant $1<p_0<\infty$ such that 
for every $w_0\in A_{p_0}$, the inequality
\[
\| f \|_{L^{p_0}(w_0)} \le C\, \| g \|_{L^{p_0}(w_0)}
\]
holds for all $f\in L^{p_0}(w_0)$ and all measurable functions $g$.
Let $p(\cdot)\in \mathcal{P}(\mathbb{R}^n)$ and $w$ be a weight.
If the Hardy--Littlewood maximal operator $M$ is bounded on
$L^{p(\cdot)}(w)$ and on $L^{p'(\cdot)}(w^{-\frac{1}{p(\cdot)-1}})$, 
then we have the inequality
\[
\| f \|_{L^{p(\cdot)}(w)}
\le C\, \| g \|_{L^{p(\cdot)}(w)}
\]
holds for all $f\in L^{p(\cdot)}(w)$ and all measurable functions $g$.
\end{theorem}

\begin{remark}
For general variable exponent $p(\cdot) \in \mathcal{P}(\mathbb{R}^n)$, 
it is not proved that the assumption $w\in A_{p(\cdot)}$ implies the equivalence of 
the following two conditions
\begin{enumerate}
\item[(a)]
$M$ is bounded on $L^{p(\cdot)}(w)$.
\item[(b)]
$M$ is bounded on $L^{p'(\cdot)}(w^{-\frac{1}{p(\cdot)-1}})$.
\end{enumerate}
If we additionally suppose that $p(\cdot) \in LH(\mathbb{R}^n)$, then
(a) is immediately true.
We note that $p(\cdot) \in LH(\mathbb{R}^n)\cap \mathcal{P}(\mathbb{R}^n)$
implies  $p'(\cdot) \in LH(\mathbb{R}^n)\cap \mathcal{P}(\mathbb{R}^n)$. 
Thus (b) is also true.
\end{remark}

Combing the two theorems above, 
we have the following boundedness of 
the intrinsic square function on weighted Lebesgue spaces with variable exponent.

\begin{corollary} 
\label{Corollary-Wilson}
Let $0<\beta \le 1$, $p(\cdot) \in LH(\mathbb{R}^n)\cap \mathcal{P}(\mathbb{R}^n)$
and $w\in A_{p(\cdot)}$.
Then the intrinsic square function $S_{\beta}$ is bounded on $L^{p(\cdot)}(w)$.
\end{corollary}

\subsection{Statement of the main result}

\begin{theorem}
Let $0<\beta \le 1$, 
$p(\cdot)\in LH(\mathbb{R}^n) \cap \mathcal{P}(\mathbb{R}^n)$, 
$0<q<\infty$, 
$1/p_-<r<1$, 
$w\in A_{rp(\cdot)}$ and 
$-n\delta<\alpha<n(1-r)$, where 
$0<\delta<1$ is a constant 
satisfying 
\[
\frac{\| \chi_{B_k} \|_{L^{p(\cdot)}(w)}}{\| \chi_{B_l} \|_{L^{p(\cdot)}(w)}}
\le
C\, 2^{\delta n(k-l)}
\]
for all $k, \, l\in \mathbb{Z}$ with $k\le l$.
Then the intrinsic square function $S_{\beta}$ is bounded on
$\dot{K}^{\alpha,q}_{p(\cdot)}(w)$ and on $K^{\alpha,q}_{p(\cdot)}(w)$.  
\end{theorem}
\begin{proof}
We prove the boundedness on the homogeneous space
$\dot{K}^{\alpha,q}_{p(\cdot)}(w)$. 
The proof similar to below is valid for the non-homogeneous space $K^{\alpha,q}_{p(\cdot)}(w)$. 
We
decompose $f\in \dot{K}^{\alpha,q}_{p(\cdot)}(w)$ as
\[
f=f\chi_{B_{k+1}\, \backslash \, B_{k-2}}
+ f\chi_{B_{k-2}}
+f\chi_{\mathbb{R}^n\, \backslash \, B_{k+1}}
.
\]
Thus we obtain
\begin{eqnarray*}
\lefteqn{ \| S_{\beta}f\|_{\dot{K}^{\alpha,q}_{p(\cdot)}(w)}
} \\
&\le& C  \Bigg\{  
\left(
\sum_{k=-\infty}^{\infty}
2^{\alpha k q} \|
S_{\beta}(f\chi_{B_{k+1}\, \backslash \, B_{k-2}}) \chi_k
\|_{L^{p(\cdot)}(w)}^q
\right)^{1/q} 
\\
& & \qquad
+ 
\left(
\sum_{k=-\infty}^{\infty}
2^{\alpha k q} \|
S_{\beta}(f\chi_{B_{k-2}}) \chi_k
\|_{L^{p(\cdot)}(w)}^q
\right)^{1/q} 
\\
& & \qquad +
\left(
\sum_{k=-\infty}^{\infty}
2^{\alpha k q} \|
S_{\beta}(f\chi_{\mathbb{R}^n\, \backslash \, B_{k+1}}) \chi_k
\|_{L^{p(\cdot)}(w)}^q
\right)^{1/q} 
\Bigg\}  
\\
&=:& C(T_1+T_2+T_3).
\end{eqnarray*}
For each $i=1, \, 2, \, 3$ we start the estimate $T_i$. 

We first consider $T_1$. 
Using the boundedness of $S_{\beta}$ on $L^{p(\cdot)}(w)$ we get 
\begin{eqnarray*}
T_1
&\le &
\left(
\sum_{k=-\infty}^{\infty}
2^{\alpha k q} \|
S_{\beta}(f\chi_{B_{k+1}\, \backslash \, B_{k-2}}) 
\|_{L^{p(\cdot)}(w)}^q
\right)^{1/q} 
\\
&\le &
C
\left(
\sum_{k=-\infty}^{\infty}
2^{\alpha k q} \|
f\chi_{B_{k+1}\, \backslash \, B_{k-2}}
\|_{L^{p(\cdot)}(w)}^q
\right)^{1/q} 
\\
&\le &
C\, \| f\|_{\dot{K}^{\alpha,q}_{p(\cdot)}(w)} .
\end{eqnarray*}

Next we estimate $T_2$.
\begin{eqnarray}
T_2
&=&
\left(
\sum_{k=-\infty}^{\infty}
2^{\alpha k q} \left\|
S_{\beta}(\sum_{l=-\infty}^{k-2}f\chi_l) \chi_k
\right\|_{L^{p(\cdot)}(w)}^q
\right)^{1/q} 
\nonumber   \\
&\le&
\left(
\sum_{k=-\infty}^{\infty}
2^{\alpha k q} 
\left(  \sum_{l=-\infty}^{k-2}
\left\|
S_{\beta}(f\chi_l) \chi_k
\right\|_{L^{p(\cdot)}(w)}
\right)^q
\right)^{1/q} .
\label{T2-first}
\end{eqnarray}
Now we take 
$k\in \mathbb{Z}$, $l\le k-2$, $x\in D_k$ and $(y,t)\in \Gamma(x)$. 
For every $\varphi \in \mathcal{C}_{\beta}$ we have
\begin{eqnarray*}
\left|  (f\chi_l)*\varphi_t(y) \right|
&=&
\left|
\int_{D_l} \varphi_t(y)f(z)\, dz
\right|
\\
&\le&
Ct^{-n} \int_{\{  z\in D_l \, : \, |y-z|<t \}} |f(z)|\, dz .
\end{eqnarray*}
Using a point $z\in D_l$ with $|y-z|<t$ we obtain
\begin{eqnarray*}
t
&=&
\frac{1}{2}(t+t)
>\frac{1}{2}(|x-y|+|y-z|)
\ge \frac{1}{2}|x-z|
\ge
\frac{1}{2}(|x|-|z|)
\\
&\ge & \frac{1}{2}(|x|-2^l)
\ge \frac{1}{2}(|x|-2^{k-2})
\ge \frac{1}{2}(|x|-2^{-1}|x|)=\frac{|x|}{4} .
\end{eqnarray*}
Hence we get 
\begin{align*}
\lefteqn{ \left| S_{\beta}(f\chi_l)(x) \right|
} \\
&= 
\left( \int \! \int_{\Gamma(x)} 
\left( \sup_{\varphi \in \mathcal{C}_{\beta}} 
|(f\chi_l)*\varphi_t(y)|^2
\frac{dy\, dt}{t^{n+1}}
\right)^2
\right)^{1/2}
\\
&\le 
C\left( \int_{\frac{|x|}{4}}^{\infty}
\int_{\{ y \, : \, |x-y|<t \}}
\left( \frac{1}{t^n} \int_{\{  z\in D_l \, : \, |y-z|<t \}} |f(z)|\, dz \right)^2
\frac{dy \, dt}{t^{n+1}}
\right)^{1/2}
\\
&\le
C\left(  \int_{D_l} |f(z)|\, dz \right)
\left( \int_{\frac{|x|}{4}}^{\infty}
\left( \int_{\{ y \, : \, |x-y|<t \}}dy \right)
\frac{dt}{t^{3n+1}}
\right)^{1/2}
\\
&=
C\left(  \int_{D_l} |f(z)|\, dz \right)
\left(  \int_{\frac{|x|}{4}}^{\infty}
\frac{dt}{t^{2n+1}}
\right)^{1/2}
\\
&=
C\left(  \int_{D_l} |f(z)|\, dz \right)
|x|^{-n} .
\end{align*}
Applying the generalized H\"older inequality and 
Lemma \ref{Lemma-Izuki-Banach}, we have
\begin{eqnarray*}
\lefteqn{ \left| S_{\beta}(f\chi_l)(x) \right|
} \\
&\le &
C\, |x|^{-n} 
\left\|  fw^{1/p(\cdot)}\chi_l \right\|_{L^{p(\cdot)}(\mathbb{R}^n)}
\cdot
\left\|  w^{-1/p(\cdot)}\chi_l \right\|_{L^{p'(\cdot)}(\mathbb{R}^n)}
\\
&\le &
C\, |x|^{-n} 
\left\|  fw^{1/p(\cdot)}\chi_l \right\|_{L^{p(\cdot)}(\mathbb{R}^n)}
\cdot
\left\|  w^{-1/p(\cdot)}\chi_{B_l} \right\|_{L^{p'(\cdot)}(\mathbb{R}^n)} 
\\
&=&
C\, |x|^{-n} 
\left\|  fw^{1/p(\cdot)}\chi_l \right\|_{L^{p(\cdot)}(\mathbb{R}^n)}
\cdot
\| \chi_{B_l} \|_{(L^{p(\cdot)}(w))'}
\\
&\le&
C\, |x|^{-n} 
\left\|  fw^{1/p(\cdot)}\chi_l \right\|_{L^{p(\cdot)}(\mathbb{R}^n)}
\cdot
|B_l| \| \chi_{B_l} \|_{L^{p(\cdot)}(w)}^{-1}.
\end{eqnarray*}
We note that $x\in D_k$ implies $|x|>2^{k-1}$, that is $|x|^{-n}<C|B_k|^{-1}$.
Thus we get
\begin{eqnarray}
\left| S_{\beta}(f\chi_l)(x) \right|
\le 
C\cdot \frac{|B_l|}{|B_k|} \| f \chi_l \|_{L^{p(\cdot)}(w)}  \| \chi_{B_l} \|_{L^{p(\cdot)}(w)}^{-1}.
\label{T2-second}
\end{eqnarray}
Combing (\ref{T2-first}) and (\ref{T2-second}), we obtain
\begin{eqnarray}
\lefteqn{ T_2
}  \nonumber \\
&\le&
C\left(
\sum_{k=-\infty}^{\infty}
2^{\alpha k q} 
\left(  \sum_{l=-\infty}^{k-2}
\frac{|B_l|}{|B_k|} 
\frac{ \| \chi_{k} \|_{L^{p(\cdot)}(w)} }{ \| \chi_{B_l} \|_{L^{p(\cdot)}(w)} }
 \| f \chi_l \|_{L^{p(\cdot)}(w)} 
\right)^q
\right)^{1/q} 
\nonumber \\
&\le&
C\left(
\sum_{k=-\infty}^{\infty}
2^{\alpha k q} 
\left(  \sum_{l=-\infty}^{k-2}
\frac{|B_l|}{|B_k|} 
\frac{ \| \chi_{B_k} \|_{L^{p(\cdot)}(w)} }{ \| \chi_{B_l} \|_{L^{p(\cdot)}(w)} }
 \| f \chi_l \|_{L^{p(\cdot)}(w)} 
\right)^q
\right)^{1/q} .
\nonumber
\end{eqnarray}
For every $k, \, l\in \mathbb{Z}$ such that $k\ge l+2$, 
we see that $B_l \subset B_k$.
We also note that $rp(\cdot) \in LH(\mathbb{R}^n) \cap \mathcal{P}(\mathbb{R}^n)$. 
Thus by virtue of Lemma \ref{Lemma-1} with $X=L^{rp(\cdot)}(w)$
we have
\begin{eqnarray*}
\frac{\| \chi_{B_k} \|_{L^{p(\cdot)}(w)} }{ \| \chi_{B_l} \|_{L^{p(\cdot)}(w)} }
=\left(
\frac{\| \chi_{B_k} \|_{L^{rp(\cdot)}(w)} }{ \| \chi_{B_l} \|_{L^{rp(\cdot)}(w)} }
\right)^r
\le C
\left(
\frac{|B_k|}{|B_l|}
\right)^r
=C\, 2^{(k-l)nr}.
\end{eqnarray*}
Hence we have
\begin{eqnarray}
\lefteqn{ T_2
} \nonumber \\
&\le & C
\left(
\sum_{k=-\infty}^{\infty}
2^{\alpha k q} 
\left(  \sum_{l=-\infty}^{k-2}
2^{n(r-1)(k-l)}
 \| f \chi_l \|_{L^{p(\cdot)}(w)} 
\right)^q
\right)^{1/q} .
\label{T2-third}
\end{eqnarray}
We shall continue the estimate 
remarking the range of $q$ carefully. 

Now let us suppose that $0<q\le 1$. 
For general non-negative sequence $\{ a_{\lambda} \}_{\lambda \in \Lambda}$, 
it is well-known that the inequality
\begin{equation}
\left( \sum_{\lambda \in \Lambda} a_{\lambda} \right)^q
\le \sum_{\lambda \in \Lambda} a_{\lambda}{}^q
\label{known-inequality}
\end{equation}
holds. 
Applying (\ref{known-inequality}) to (\ref{T2-third}), we have
\begin{eqnarray*}
T_2
&\le &
C\left(
\sum_{k=-\infty}^{\infty}
2^{\alpha k q} 
  \sum_{l=-\infty}^{k-2}
2^{qn(r-1)(k-l)}
 \| f \chi_l \|_{L^{p(\cdot)}(w)}^q 
\right)^{1/q}
\\
&=&
C\left(   \sum_{l=-\infty}^{\infty}
2^{\alpha l q} \| f \chi_l \|_{L^{p(\cdot)}(w)}^q
\sum_{k=l+2}^{\infty} 
 2^{q(k-l)(\alpha -n(1-r))}
\right)^{1/q}
\\
&=&
C\left(   \sum_{l=-\infty}^{\infty}
2^{\alpha l q} \| f \chi_l \|_{L^{p(\cdot)}(w)}^q
\right)^{1/q}
=C\, \| f \|_{\dot{K}^{\alpha,q}_{p(\cdot)}(w)} ,
\end{eqnarray*}
where we have used the condition that $\alpha<n(1-r)$.

We next consider the case $1<q<\infty$.
Because we have supposed that $\alpha<n(1-r)$, we can 
take a constant $1<s<\infty$ so that $\alpha<\frac{n}{s}(1-r)$.
Using the usual H\"older inequality, for every $k\in \mathbb{Z}$ we get
\begin{eqnarray*}
\lefteqn{  
\left(  \sum_{l=-\infty}^{k-2}
2^{n(r-1)(k-l)}
 \| f \chi_l \|_{L^{p(\cdot)}(w)} 
\right)^q
}
\\
&=&
\left(  \sum_{l=-\infty}^{k-2}
2^{(l-k)\cdot \frac{n}{s}(1-r)} 2^{-\alpha l}
\cdot
2^{(l-k)\cdot \frac{n}{s'}(1-r)}
2^{\alpha l}
 \| f \chi_l \|_{L^{p(\cdot)}(w)} 
\right)^q
\\
&\le &
\left(  \sum_{l=-\infty}^{k-2}
2^{((l-k)\cdot \frac{n}{s}(1-r)-\alpha l)q'}
\right)^{q/q'}
\\
& & \times
\left(  \sum_{l=-\infty}^{k-2}
2^{(l-k)\cdot \frac{nq}{s'}(1-r)}
2^{\alpha lq}
 \| f \chi_l \|_{L^{p(\cdot)}(w)}^q 
\right)
\\
&=&
\left(  \sum_{l=-\infty}^{k-2}
2^{(l(\frac{n}{s}(1-r)-\alpha)-k\cdot \frac{n}{s}(1-r))q'}
\right)^{q/q'}
\\
& & \times
\left(  \sum_{l=-\infty}^{k-2}
2^{(l-k)\cdot \frac{nq}{s'}(1-r)}
2^{\alpha lq}
 \| f \chi_l \|_{L^{p(\cdot)}(w)}^q 
\right)
\\
&=&
C\cdot 2^{-k\alpha q} \left(  \sum_{l=-\infty}^{k-2}
2^{(l-k)\cdot \frac{nq}{s'}(1-r)}
2^{\alpha lq}
 \| f \chi_l \|_{L^{p(\cdot)}(w)}^q 
\right) .
\end{eqnarray*}
Applying this estimate to (\ref{T2-third}) we get
\begin{align*}
\lefteqn{ T_2
} \\
&\le C
\left(
\sum_{k=-\infty}^{\infty}
2^{\alpha k q} 
\cdot 2^{-k\alpha q} \left(  \sum_{l=-\infty}^{k-2}
2^{(l-k)\cdot \frac{nq}{s'}(1-r)}
2^{\alpha lq}
 \| f \chi_l \|_{L^{p(\cdot)}(w)}^q 
\right) 
\right)^{1/q} 
\\
&=
C
\left(
\sum_{l=-\infty}^{\infty}
2^{\alpha l q} \| f \chi_l \|_{L^{p(\cdot)}(w)}^q 
\sum_{k=l+2}^{\infty}
2^{(l-k)\cdot \frac{nq}{s'}(1-r)}
\right)^{1/q} 
\\
&=
C
\left(
\sum_{l=-\infty}^{\infty}
2^{\alpha l q} \| f \chi_l \|_{L^{p(\cdot)}(w)}^q 
\right)^{1/q} 
\\
&=C\, \| f \|_{\dot{K}^{\alpha,q}_{p(\cdot)}(w)}.
\end{align*}

Finally we estimate $T_3$. 
We note that
\begin{eqnarray*}
T_3\le
\left(
\sum_{k=-\infty}^{\infty}
2^{\alpha k q} 
\left( \sum_{l=k+2}^{\infty}
\|
S_{\beta}(f\chi_l) \chi_k
\|_{L^{p(\cdot)}(w)}
\right)^q
\right)^{1/q} .
\end{eqnarray*}
For every $k\in \mathbb{Z}$, $x\in D_k$, $l\ge k+2$, 
$(y,t)\in \Gamma(x)$ and 
$z\in D_l$ with $|y-z|<t$, we see that
\begin{eqnarray*}
t&=&
\frac{1}{2}(t+t)
>\frac{1}{2}(|x-y|+|y-z|)
\ge \frac{1}{2}|x-z|
\ge \frac{1}{2}(|z|-|x|)
\\
&>&
\frac{1}{2}(2^{l-1}-2^k)\ge 2^{l-3} .
\end{eqnarray*}
Thus we have
\begin{eqnarray*}
\lefteqn{ |S_{\beta}(f\chi_l)(x)|
} \\
&=&
\left(
\int \! \int_{\Gamma(x)}
\left( \sup_{\varphi \in \mathcal{C}_{\beta}}
|(f\chi_l)*\varphi_t(y)|
\right)^2
\frac{dy\, dt}{t^{n+1}}
\right)^{1/2}
\\
&\le&
C
\left(
\int \! \int_{\Gamma(x)}
\left( 
t^{-n}
\int_{\{ z\in D_l \, : \, |y-z|<t \}}
|f(z)|\, dz
\right)^2
\frac{dy\, dt}{t^{n+1}}
\right)^{1/2}
\\
&\le&
C
\left(
\int_{2^{l-3}}^{\infty}\int_{\{ y\, : \, |x-y|<t \}}
t^{-3n-1}
\left( 
\int_{D_l}
|f(z)|\, dz
\right)^2
dy\, dt
\right)^{1/2}
\\
&=&
C\left( 
\int_{D_l}
|f(z)|\, dz
\right)
\left(
\int_{2^{l-3}}^{\infty}
t^{-2n-1}
\, dt
\right)^{1/2}
\\
&=&
C\, |B_l|^{-1}
\int_{D_l}
|f(z)|\, dz .
\end{eqnarray*}
By virtue of the generalized H\"older inequality and 
Lemma \ref{Lemma-Izuki-Banach}, we have
\begin{align*}
|S_{\beta}(f\chi_l)(x)|
&\le 
C\, |B_l|^{-1}
\| fw^{1/p(\cdot)}\chi_l \|_{L^{p(\cdot)}(\mathbb{R}^n)}
\| w^{-1/p(\cdot)}\chi_l \|_{L^{p'(\cdot)}(\mathbb{R}^n)}
\\
&\le 
C\, 
\| fw^{1/p(\cdot)}\chi_l \|_{L^{p(\cdot)}(\mathbb{R}^n)}
\cdot |B_l|^{-1}
\| w^{-1/p(\cdot)}\chi_{B_l} \|_{L^{p'(\cdot)}(\mathbb{R}^n)}
\\
&\le 
C\, 
\| f\chi_l \|_{L^{p(\cdot)}(w)}
\cdot
\| \chi_{B_l} \|_{L^{p(\cdot)}(w)}^{-1}.
\end{align*}
Hence we obtain
\begin{eqnarray}
\lefteqn{ T_3 } \nonumber \\
&\le&
C\left(
\sum_{k=-\infty}^{\infty}
2^{\alpha k q} 
\left( \sum_{l=k+2}^{\infty}
\| f\chi_l \|_{L^{p(\cdot)}(w)}
\frac{ \| \chi_{k} \|_{L^{p(\cdot)}(w)} }
{ \| \chi_{B_l} \|_{L^{p(\cdot)}(w)} }
\right)^q
\right)^{1/q} 
\nonumber  \\
&\le &
C\left(
\sum_{k=-\infty}^{\infty}
2^{\alpha k q} 
\left( \sum_{l=k+2}^{\infty}
\| f\chi_l \|_{L^{p(\cdot)}(w)}
\frac{ \| \chi_{B_k} \|_{L^{p(\cdot)}(w)} }
{ \| \chi_{B_l} \|_{L^{p(\cdot)}(w)} }
\right)^q
\right)^{1/q} 
\nonumber  \\
&\le &
C\left(
\sum_{k=-\infty}^{\infty}
2^{\alpha k q} 
\left( \sum_{l=k+2}^{\infty}
\| f\chi_l \|_{L^{p(\cdot)}(w)}
2^{\delta n (k-l)}
\right)^q
\right)^{1/q} .
\label{T3-first}
\end{eqnarray}
In order to continue the estimate for $T_3$ 
we consider the two cases $0<q\le 1$ and $1<q<\infty$ respectively. 

We first assume that $0<q\le 1$.
Applying inequality (\ref{known-inequality}) again to (\ref{T3-first}), we get
\begin{eqnarray*}
T_3
&\le&
C\left(
\sum_{k=-\infty}^{\infty}
2^{\alpha k q} 
 \sum_{l=k+2}^{\infty}
\| f\chi_l \|_{L^{p(\cdot)}(w)}^q
2^{\delta n q (k-l)}
\right)^{1/q} 
\nonumber  \\
&=&
C\left(
\sum_{l=-\infty}^{\infty}
2^{\alpha l q} \| f\chi_l \|_{L^{p(\cdot)}(w)}^q
 \sum_{k=-\infty}^{l-2}
2^{q(\alpha +n\delta) (k-l)}
\right)^{1/q} 
\\
&=&
C\left(
\sum_{l=-\infty}^{\infty}
2^{\alpha l q} \| f\chi_l \|_{L^{p(\cdot)}(w)}^q
\right)^{1/q} 
=C\, \| f \|_{\dot{K}^{\alpha,q}_{p(\cdot)}(w)} ,
\end{eqnarray*}
where we have used the condition that $-n\delta<\alpha$.

Finally we estimate $T_3$ in the case $1<q<\infty$.
We can take a constant $1<u<\infty$ so that $\alpha +n\delta/u>0$ 
because we have supposed that $-n\delta<\alpha$.
Using the usual H\"older inequality, 
for each $k\in \mathbb{Z}$ we get
\begin{eqnarray}
\lefteqn{
\left( \sum_{l=k+2}^{\infty}
\| f\chi_l \|_{L^{p(\cdot)}(w)}
2^{\delta n (k-l)}
\right)^q
}  \nonumber  \\
&=&
\left( \sum_{l=k+2}^{\infty}
2^{\alpha l}
\| f\chi_l \|_{L^{p(\cdot)}(w)}
2^{\delta n (k-l)/u'}
\cdot
2^{-\alpha l} 2^{\delta n (k-l)/u}
\right)^q
\nonumber  \\
&\le &
\left( \sum_{l=k+2}^{\infty}
2^{\alpha l q}
\| f\chi_l \|_{L^{p(\cdot)}(w)}^q
2^{q \delta n (k-l)/u'}
\right)
\nonumber  \\
& & \qquad \times
\left( \sum_{l=k+2}^{\infty}
2^{-\alpha l q'} 2^{q' \delta n (k-l)/u}
\right)^{q/q'}
\nonumber \\
&=&
\left( \sum_{l=k+2}^{\infty}
2^{\alpha l q}
\| f\chi_l \|_{L^{p(\cdot)}(w)}^q
2^{q \delta n (k-l)/u'}
\right) \nonumber \\
& &\qquad \times
\left(  2^{q' \delta n k/u}
\sum_{l=k+2}^{\infty}
2^{- l q'(\alpha +n\delta/u)} 
\right)^{q/q'}
\nonumber \\
&=&
C\cdot 2^{-\alpha k q}
\left( \sum_{l=k+2}^{\infty}
2^{\alpha l q}
\| f\chi_l \|_{L^{p(\cdot)}(w)}^q
2^{q \delta n (k-l)/u'}
\right) .
\label{T3-second}
\end{eqnarray}
Applying (\ref{T3-second}) to (\ref{T3-first}) we obtain
\begin{eqnarray*}
\lefteqn{ 
T_3
}\\
&\le&
C\left(
\sum_{k=-\infty}^{\infty}
2^{\alpha k q} 
2^{-\alpha k q}
\left( \sum_{l=k+2}^{\infty}
2^{\alpha l q}
\| f\chi_l \|_{L^{p(\cdot)}(w)}^q
2^{q \delta n (k-l)/u'}
\right) 
\right)^{1/q}
\\
&=&
C\left(
\sum_{l=-\infty}^{\infty}
2^{\alpha l q} \| f\chi_l \|_{L^{p(\cdot)}(w)}^q
 \sum_{k=-\infty}^{l-2}
2^{q \delta n (k-l)/u'}
\right)^{1/q}
\\
&=&
C\left(
\sum_{l=-\infty}^{\infty}
2^{\alpha l q} \| f\chi_l \|_{L^{p(\cdot)}(w)}^q
\right)^{1/q}
=C\, \| f \|_{\dot{K}^{\alpha,q}_{p(\cdot)}(w)} .
\end{eqnarray*}

Consequently we have finished all estimates and proved the theorem.
\end{proof}

\section*{Acknowledgement} 
The first author was partially supported by Grand-in-Aid for Scientific Research  (C), 
No. 15K04928, for Japan Society for the Promotion of Science. 
The second author was partially supported by Grand-in-Aid for Scientific Research  (C), 
No. 16K05212, for Japan Society for the Promotion of Science.

\end{document}